\documentclass[12pt]{amsart}

\usepackage[matrix,arrow,curve,frame]{xy}
\usepackage{amsmath,amsthm,amssymb,enumerate}
\usepackage{latexsym}
\usepackage{amscd}
\usepackage{hyperref}
\usepackage{euscript}
\usepackage{fullpage}
\usepackage{color}
\usepackage{tikz}
\usepackage{tikz-cd}
\usepackage[all]{xypic}

\usepackage{mathptmx,enumitem,cite,array}

\usepackage[dvips]{geometry}

\newtheorem{theorem}{Theorem}[section]
\newtheorem{lemma}[theorem]{Lemma}

\newtheorem{proposition}[theorem]{Proposition}

\newtheorem{example}[theorem]{Example}
\theoremstyle{example}
\newtheorem{remark}[theorem]{Remark}

\numberwithin{equation}{section}

\newcommand{\beq}{\begin{equation}}
\newcommand{\eeq}{\end{equation}}
\DeclareMathOperator{\ind}{index}

\newcommand{\ZZ}{\mathbb{Z}}
\newcommand{\RR}{\mathbb{R}}
\newcommand{\CC}{\mathbb{C}}

\DeclareMathOperator{\tr}{tr}

\DeclareMathOperator{\DInd}{D-Ind}
\DeclareMathOperator{\DsInd}{D^{spin}-Ind}

\DeclareMathOperator{\DInda}{D_1-Ind}
\DeclareMathOperator{\DIndb}{D_2-Ind}
\DeclareMathOperator{\DsIndb}{D^{spin}_2-Ind}

\DeclareMathOperator{\Ind}{Ind}

\newcommand{\Z}{\mathbb{Z}}

\newcommand{\pp}{\mathfrak{p}}
\newcommand{\ppc}{\mathfrak{p}_\CC}

\newcommand{\dirac}{\partial\!\!\!\!\!\slash}

\newcommand {\be}{\begin{equation}}
\newcommand {\ee}{\end{equation}}
\newcommand{\h}{\begin{eqnarray*}}
\newcommand{\e}{\end{eqnarray*}}

\begin{document}


\title
{Witten Genus and Elliptic genera for proper actions}

\author{Fei Han}
\address{Department of Mathematics,
National University of Singapore, Singapore 119076}
\email{mathanf@nus.edu.sg}

 \author{Varghese Mathai}
\address{School of Mathematical Sciences,
University of Adelaide, Adelaide 5005, Australia}
\email{mathai.varghese@adelaide.edu.au}

\subjclass[2010]{Primary 58J26, Secondary 57S20, 53C27, 11F55, 19K35}
\keywords{Elliptic genera, Witten genus, proper actions, loop Dirac induction, modularity, group $C^*$-algebras, representation ring, operator K-theory}
\date{}

\maketitle

\begin{abstract}
In this paper, we construct for the first time, the Witten genus and elliptic genera on noncompact manifolds with a proper cocompact action by an almost connected Lie group and prove  vanishing and rigidity results that generalise known results for compact group actions on compact manifolds. We also compute our genera for some interesting examples.
\end{abstract}

\tableofcontents

\section*{Introduction}

In 1970, Atiyah and Hirzebruch \cite{AH70} proved a striking result, showing that if a compact connected group of positive dimension acts 
non-trivially on a compact Spin manifold, then the equivariant index of the Spin Dirac operator vanishes,
and in particular, the $\hat A$ genus of the compact manifold also vanishes.
In 1980's, Witten studied two-dimensional quantum field theories 
and the index of Dirac operator in free loop spaces. 
In \cite{W87}, Witten argued that the partition function of a type II superstring as a function depending on the modulus 
of the worldsheet elliptic curve, is an elliptic genus.
In \cite{W86}, Witten derived a series of
twisted Dirac operators from the free loop space $LM$ on a compact Spin manifold $M$.
The elliptic genera constructed by Landweber-Stong \cite{LS88}
and Ochanine \cite{O87} in a topological way
turn out to be the indices of these elliptic operators.
Motivated by physics, Witten conjectured that these elliptic operators should be rigid. The Witten conjecture was first proved by Taubes \cite{T89} and Bott-Taubes \cite{BT89}. In \cite{Liu96}, using the modular invariance property, Liu
presented a simple and unified proof of the Witten conjecture. In \cite{Liu95}, continuing the Witten-Bott-Taubes-Liu rigidity theorem,  Liu discovered vast generalizations, in particular a profound vanishing theorem for the Witten genus under the condition that the equivariant first Pontryagin class $p_1(M)_{S^1}=n\cdot \pi^* u^2$, which, as observed by Dessai \cite{De1}, when the $S^1$-action is induced from an $S^3$-action, is equivalent to that the manifold is string, i.e. the free loop space is Spin \cite{ML}. Later Liu-Ma \cite{LM1, LM2} and Liu-Ma-Zhang \cite{LMZ1, LMZ2}
generalized the rigidity and vanishing theorems to the family case
 on the levels of equivariant Chern character and of equivariant $K$-theory. On the other hand, in algebraic geometry,  Totaro\cite{T00}  used the rigidity theorem for complex elliptic genus to study the question of which characteristic numbers can be defined on compact complex algebraic varieties with singularities.

In \cite{HM16}, Hochs and the second author extended the Atiyah-Hirzebruch theorem
in another direction,  to the non-compact setting. 
More precisely, let $M$ be a complete Riemannian manifold of even dimension, 
on which an almost connected Lie group $G$ acts properly and isometrically. 
Suppose $M/G$ is compact and $M$ has a $G$-equivariant Spin-structure. Suppose 
$G/K$ is even dimensional and $G$-Spin, where $K$ is a maximal compact
subgroup of $G$. Let 
\[
\ind_G(\dirac_M)  \in K_{0}(C^*_rG)
\]
 be the equivariant index of the associated Spin-Dirac operator. 
 Here $K_{0}(C^*_rG)$ is the $K$-theory of the reduced group $C^*$-algebra of $G$, and $\ind_G$ denotes the analytic assembly map used in the Baum--Connes conjecture \cite{BCH94}, \cite{Kas88}.
Atiyah and Hirzebruch's vanishing result generalises as follows in \cite{HM16}.
If there is a point in $M$ whose stabiliser in $G$  
 is not a maximal compact subgroup of $G$ (this condition is called {\it properly nontrivial}) and $K$ has positive dimension, then
\[
\ind_G(\dirac_M) = 0  \in K_{0}(C^*_rG).
\]
The following fact follows immediately from the theorem above and Theorem 6.12 proved by Wang in \cite{HW}.
Under these hypotheses, one has
\begin{equation}\label{iformula} 
0=\tau_*(\ind_G(\dirac_M)) =\int_{M} c\cdot \hat{A}(M),
\end{equation}
 where  $c\in C^{\infty}_c(M)$ is a cutoff function, that is a non-negative function satisfying 
\be \label{cutoff}
\int_Gc(g^{-1}m)\,\mathrm{d}g=1,
\ee
and $\tau : C^*_r(G) \to \mathbb C$ denotes 
the von Neumann trace. The right hand side of \eqref{iformula} is independent of the choice of cut-off function $c$ and can be viewed the ``averaged $\hat A$ genus".

In this paper, we generalise the definitions of the Witten genus and 
elliptic genera to the situation of proper cocompact Lie group actions on noncompact manifolds and study their properties. In particular, we establish vanishing and rigidity properties of
the equivariant Witten genus and the equivariant elliptic genera respectively. We also compute our genera for some interesting examples to illustrate the difference between these  noncompact genera and the usual genera.

This would be a significant infinite dimensional generalisation of \cite{HM16} and should lead to significant advances. Atiyah and Hirzebruch's vanishing theorem for the $\hat A$ genus on compact Spin manifolds with a nontrivial action of a compact connected Lie group, sparked widespread international interest, especially after Witten
generalised their result to the rigidity of elliptic genera and Liu's discovery of vanishing theorem for Witten genus, which can be viewed as infinite dimensional analogs of the Atiyah-Hirzebruch theorem. Our
vanishing and rigidity theorems in this paper is the first result of this type in the noncompact world.
The innovation in our paper is to significantly generalise the method {\em quantisation commutes with induction} 
to cover the case of the Witten genus and elliptic genera. More precisely, we establish {\em quantisation commutes with twisted induction} diagrams, generalising quantisation commutes with induction theorems in the literature. 

Let $G$ be an almost connected Lie group and  suppose that $G$
acts properly and cocompactly on a smooth manifold $M$. Let $K$ be a maximal compact
subgroup of $G$. A theorem of Abels \cite{Abels} shows that $M$ can be realised as a fibre bundle over $G/K$ with fibre a compact $K$-manifold $N$. That is, $M=G\times_K N$. 
Suppose $G/K$ has a $G$-equivariant Spin-structure and $M$ has a $G$-equivariant Spin-structure, then by the 2 out of 3-lemma, $N$ has a $K$-equivariant Spin structure, and conversely. From now on, we will assume that $\dim M$ and $\dim(G/K)$ are even. Let 
$\mathfrak p$ be a complementary
subspace for the Lie algebra $\mathfrak k$ of $K$ in the  Lie algebra
$\mathfrak g$ of $G$  that it is invariant for the
adjoint action of $K$ and we endow $\mathfrak p$ with a $K$-invariant
euclidian metric. The above assumption means that the homomorphism $K\to
\mathrm{SO}(\mathfrak p)$ lifts to $\mathrm{Spin}(\mathfrak p)$. Let $R(K)$ be the complex representation ring of $K$. Then one has the {\em quantisation commutes with induction} diagram \cite{HM16, GMW}: 
\be \label{qi}
\xymatrix{
K_{0}^G(M) \ar[rrrr]^-{ \ind_G} & & & & K_{0}(C^*_{r}G) \\
K_{0}^K(N) \ar[u]^-{K-{\rm Ind}_K^G} \ar[rrrr]^-{\ind_K} & & &
 & R(K)  \ar[u]_-{\DInd_{K}^{G}
 } 
},
\ee
where $K_{0}^G(M)$ denotes the equivariant K-homology of $M$, the left vertical arrow is analytic induction from $K$ to $G$ and the right vertical arrow is the {\it Dirac induction} from $K$ to $G$. See \cite{HM16, GMW} for the details. 

Our proof of the vanishing of Witten genus essentially requires us to establish the commutativity
of the following diagram:
\be
\xymatrix{
K_{0}^G(M) \ar[rrrr]^-{ \ind_G(\bullet \,\,\, \otimes \Theta(T_\CC M,\tau))} & & & & K_{0}(C^*_{r}G) [[q]]\\
K_{0}^K(N) \ar[u]^-{K-{\rm Ind}_K^G} \ar[rrrr]^-{\ind_K(\bullet \,\,\, \otimes \Theta(T_\CC N, \tau))} & & &
 & R(K) [[q]] \ar[u]_-{\DInd_{LK}^{LG}
 }
 },
\ee
where $\DInd_{LK}^{LG}(\bullet) = \ind_G(\dirac_{G/K} \otimes \Theta(\ppc, \tau)\otimes\,\,\,\bullet)$ is a loop version of Dirac induction. 
Our proof of the rigidity of the elliptic genera essentially requires us to establish the commutativity
of the following diagrams:
\be 
\xymatrix{
K_{0}^G(M) \ar[rrrrrr]^-{ \ind_G(\bullet \,\,\, \otimes \Delta^+(TM)\oplus\Delta^-(TM))_\CC\otimes\Theta_1(T_\CC M,\tau))} & & & & & & K_{0}(C^*_{r}G) [[q^{1\over2}]]\\
K_{0}^K(N) \ar[u]^-{K-{\rm Ind}_K^G} \ar[rrrrrr]^-{\ind_K(\bullet \,\,\, \otimes \Delta^+(TN)\oplus\Delta^-(TN))_\CC\otimes \Theta_1(T_\CC N, \tau))} & & & & &
 & R(K) [[q^{1\over2}]] \ar[u]_-{\DInda_{LK}^{LG}
 } 
}
\ee
and 
\be 
\xymatrix{
K_{0}^G(M) \ar[rrrr]^-{ \ind_G(\bullet \,\,\, \otimes\Theta_2(T_\CC M,\tau))} & & & & K_{0}(C^*_{r}G) [[q^{1\over2}]]\\
K_{0}^K(N) \ar[u]^-{K-{\rm Ind}_K^G} \ar[rrrr]^-{\ind_K(\bullet \,\,\, \otimes \Theta_2(T_\CC N, \tau))} & & &
 & R(K) [[q^{1\over2}]] \ar[u]_-{\DIndb_{LK}^{LG}
 } 
},
\ee
where $\DInda_{LK}^{LG}$ is a loop version of Dirac induction given explicitly by
$$\ind_G(\dirac_{G/K}\otimes(\Delta^+(\mathfrak p)\oplus\Delta^-(\mathfrak p))_\CC\otimes\Theta_1({\mathfrak p}_\CC, \tau)\otimes\,\,\,\bullet)$$ and $\DIndb_{LK}^{LG}$ is a loop version of Dirac induction given explicitly by $$\ind_G(\dirac_{G/K} \otimes \Theta_2(\ppc, \tau)\otimes \,\,\,\bullet).$$ Here $\Delta^+(TM)\oplus\Delta^-(TM)$ stands for the space of the spinor bundle and similarly for $TN$ and $\pp$. The constructions of $\Theta, \Theta_j, j=1, 2$ are given in in Section 2. These diagrams generalise diagram (\ref{qi}) without $q$.

The paper is organized as follows. In Section 1, we give some preliminary concepts and results used in the paper. In Section 2, we give the construction of the Witten genus and elliptic genus for the proper actions, both the averaged version and the equivariant version. We present our main results about these genera, including modularity, a miraculous cancellation formula as well as vanishing and rigidity results. We also present the explicit formula for the genera for a class of interesting examples related to $G=SL(2, \RR)$ and leave the computation to the last section. In Section 3, we give the proof of the vanishing theorem. We prove the rigidity theorems in Section 4. In Section 5, we give the detailed computation for the genera of the examples. In Appendix, we study noncompact equivariant Witten genus and equivariant elliptic genera in the situation that $G/K$ is not $G$-Spin.

$\, $

\noindent{\em Acknowledgements.}
Fei Han was partially supported
by the grant AcRF R-146-000-218-112 from National University
of Singapore. Varghese Mathai was partially
supported by funding from the Australian Research Council, through the Australian Laureate Fellowship FL170100020. He also thanks 
Hao Guo and Hang Wang for clarifications on \cite{GMW}. Both authors thank Yanli Song for pointing out an error in a conjecture in the previous version of the paper.


\section{Preliminaries}


This section contains some preliminary concepts and results used in the paper.

Let $G$ be a locally compact, almost connected Lie group. In this paper, we will use the reduced $C^*$-algebra of $G$, denoted $C^*_r(G)$, 
which is the completion in the operator norm of the convolution algebra of integrable functions  $L^1(G)$ with respect to a Haar measure, 
viewed as an algebra of bounded operators on $L^2(G)$.

The $K$-theory of $C^*_r(G)$, denoted $K_0(C^*_r(G))$, is by fiat the Grothendieck group generated by stable equivalence classes of finitely generated modules over $C^*_r(G)$.

If $K$ is a compact connected Lie group, denote by $R(K)$ the representation ring of $K$, which is the free abelian group generated by 
equivalence classes of irreducible representations of $K$. It is well known that $R(K) \cong K_0(C^*_r(K))$. 

For example (cf. \cite{Valette}), when $G=SL(2, \RR)$, the reduced group $C^*$-algebra $C^*_r(G)$,  is Morita equivalent to 
$$
C_0(\RR/\ZZ_2) \bigoplus C_0(\RR) \rtimes \ZZ_2 \bigoplus_{n\in \ZZ\setminus \{0\}} \CC.
$$
In the last term, there is a copy of $\CC$ for each discrete series representation of $G$. Using the fact that $K$-theory
is Morita invariant and that the second term contributes a factor of $\ZZ$ whereas the first term doesn't contribute to $K$-theory, one has
$$
K_0(C^*_r(G)) \cong \bigoplus_{n\in \ZZ} \ZZ [n], 
$$
where $\ZZ[n]\cong\ZZ$.

There is a canonical morphism, $\DInd_K^G : R(K) \longrightarrow K_0(C^*_r(G))$ called {\em Dirac induction}, defined as follows. Assume that
$G/K$ is $G$-Spin, which is always the case for a double cover. Given a unitary representation $\rho:K \to U(V)$ of $K$, form the 
$G$-vector bundle $V_\rho = G \times_K V \to G/K$ over $G/K$. Let $\dirac_{G/K}$ denote the Dirac operator on $G/K$. Then 
$$\DInd_K^G (\rho) = \ind_G (\dirac_{G/K} \otimes V) \in K_0(C^*_r(G)).$$
The central {\em Connes-Kasparov  conjecture} (cf. \cite{BCH94}, \cite{Kas83}) states that the Dirac induction is an {\em isomorphism}.
It has been proved for reductive Lie groups by \cite{Lafforgue}, \cite{LafforgueICM}, \cite{wassermann} and in general 
by \cite{CEN}.

Let $G$ act properly, cocompactly on a manifold $M$.
By a result in \cite{Phillips}, for any almost connected Lie group $G$, the equivariant $K$-theory $K^0_G(M)$ is the Grothendieck
group generated by stable equivalence classes of  finite dimensional $G$-vector bundles over $M$. 
We will be assuming that $M$ is either $G$-Spin or $G$-Spin$^c$. In either case, there is a Poincar\'e duality isomorphism in  $K$-theory (cf. \cite{GMW}),
$$
K^0_G(M) \longrightarrow K_0^G(M), \qquad E \longrightarrow \dirac \otimes E,
$$
where $K_0^G(M)$ denotes the equivariant $K$-homology of $M$, which is therefore generated by twisted Dirac operators.

The equivariant index is defined by Kasparov using bivariant $K$-theory machinery. 
Using induction to the twisted crossed product (cf. \cite{Kas88})
$$
j_G : K_0^G(M) \longrightarrow KK_0(C_0(M) \rtimes G, C^*_r(G)), 
$$
and a cutoff function $c$ on $M$ (see (\ref{cutoff})), it defines an idempotent $[c] \in  KK_0(\CC, C_0(M) \rtimes G)$, the equivariant index is by fiat
the Kasparov intersection product  (cf. \cite{Kas88}),
$$
\ind_G(\dirac \otimes E) = [c] \otimes_{C_0(M) \rtimes G} j_G([\dirac \otimes E]) \in K_0(C^*_r(G)).
$$


\section{Witten genus and elliptic genera for proper actions}


In this section, we introduce the Witten genus and the elliptic genera for proper actions and present the main results about them to be proved in the next sections. 

\subsection{Witten genus and elliptic genera for proper actions: modularity}

$\, $

Let $G$ be an almost connected Lie group and  suppose that $G$
acts properly and cocompactly on a $4k$ dimensional manifold $M$. 

If $E$ is a complex vector bundle over $M$, set
$\widetilde{E}=E-\mathbb{C}^{\mathrm{rk}(E)}. $ 
Recall that for an indeterminate $t$, \be \Lambda_t(E)=\CC
|_M+tE+t^2\Lambda^2(E)+\cdots,\ \ \ S_t(E)=\CC |_M+tE+t^2
S^2(E)+\cdots, \ee are the total exterior and
symmetric powers of $E$ respectively. The following relations between these two operations hold (c.f. \cite{A67}), 
\be S_t(E)=\frac{1}{\Lambda_{-t}(E)},\ \ \ \
\Lambda_t(E-F)=\frac{\Lambda_t(E)}{\Lambda_t(F)}.
\ee

Let $q=e^{2 \pi \sqrt{-1}\tau}$ with $\tau \in \mathbb{H}$, the
upper half complex plane.

Introduce three elements (\cite{W86}) in
$K(M)[[q^{1\over2}]]$, which consist of formal power series in
$q^{1\over2}$ with coefficients in the $K$-group of $M$:

\be \Theta(T_\CC M, \tau)=\bigotimes_{n=1}^\infty
S_{q^n}(\widetilde{T_\CC M}),\ee

\be
\Theta_1(T_\CC M, \tau)=\bigotimes_{n=1}^\infty S_{q^n}(\widetilde{T_\CC
M}) \otimes \bigotimes_{m=1}^\infty \Lambda_{q^m}(\widetilde{T_\CC
M}),\ee

\be \Theta_2(T_\CC M,\tau)=\bigotimes_{n=1}^\infty
S_{q^n}(\widetilde{T_\CC M}) \otimes \bigotimes_{m=1}^\infty
\Lambda_{-q^{m-{1\over2}}}(\widetilde{T_\CC M}).\ee

One can formally expand these elements into Fourier series,

\be \Theta(T_\CC M, \tau)=W_0(T_\CC M)+W_1(T_\CC M)q+\cdots,
\ee

\be \Theta_1(T_\CC M, \tau)=A_0(T_\CC M)+A_1(T_\CC
M)q^{1\over2}+\cdots,\ee

\be \Theta_2(T_\CC M,\tau)=B_0(T_\CC M)+B_1(T_\CC M)q^{1\over2}+\cdots.
\ee

Let $\nabla^{TM}$ be a $G$-invariant connection on $TM$ and $R^{TM}=(\nabla^{TM})^2$ be the curvature of $\nabla^{TM}$. Suppose $E$ is a $G$-equivariant complex vector bundle over $M$, $\nabla^E$ is a $G$-invariant connection on $E$ and $R^{E}$ is the curvature of $\nabla^E$. 

Define 
\be \hat{A}^c(M;E) = \displaystyle \int_{M} c\cdot \hat{A}(M)\mathrm{ch}(E),\ee 
\be \hat{L}^c(M;E) = \displaystyle \int_{M} c\cdot \hat{L}(M)\mathrm{ch}(E),\ee 
where
$$\hat{A}(M)={\det}^{1/2}\left({{\sqrt{-1}\over
4\pi}R^{TM} \over \sinh\left({ \sqrt{-1}\over
4\pi}R^{TM}\right)}\right)$$
and $$\hat{L}(M)={\det}^{1/2}\left({{\sqrt{-1}\over 2\pi}R^{TM} \over \tanh\left({
\sqrt{-1}\over 4\pi}R^{TM}\right)}\right)$$ are the Hirzebruch $\hat A$-form and $\hat L$-form respectively and 
$$\mathrm{ch}(E)=\tr\left[\exp\left({\sqrt{-1}\over 2\pi}R^E\right)\right]$$ is the Chern character form.

$\hat{A}^c(M;E)$ and  $\hat{L}^c(M;E)$ are independent of the choice of the cutoff function $c$ and the connections \cite{HW}.

The virtual bundles $W_i(T_\CC M), A_i(T_\CC M)$ and $B_i(T_\CC M)$ carry connections induced from $\nabla^{TM}$. 

Define the {\it Witten genus of $(M, G)$} by
\be 
\varphi_W^c(M,\tau) = \hat{A}^c(M; \Theta(T_\CC M, \tau)) \in \RR[[q]].
\ee
That is, 
\be 
\varphi_W^c(M,\tau) =  \hat{A}^c(M) + \hat{A}^c(M;W_1(T_\CC M)) q + \cdots \in \RR[[q]].
\ee

Define the {\it elliptic genera of $(M, G)$}  by
\be 
\varphi_1^c(M, \tau) = \hat{L}^c(M; \Theta_1(T_\CC M, \tau)) \in \RR[[q^{1\over2}]],
\ee
\be 
\varphi_2^c(M,\tau) = \hat{A}^c(M; \Theta_2(T_\CC M, \tau)) \in \RR[[q^{1\over2}]].
\ee
That is, 
\be 
\varphi_1^c(M,\tau) = \hat{L}^c(M) + \hat{L}^c(M; A_1(T_\CC M)) q^{1\over2} + \cdots \in \RR[[q^{1\over2}]],
\ee 
\be 
\varphi_2^c(M,\tau) = \hat{A}^c(M) + \hat{A}^c(M; B_1(T_\CC M)) q^{1\over2} + \cdots \in \RR[[q^{1\over2}]].
\ee 

These genera can be viewed as ``averaged Witten genus and averaged elliptic genera" of $M$. To illustrate the difference between these averaged genera and the usual genera, we explicitly compute them in the following examples. 

\begin{example} \label{example} Let $S^1$ act on the complex projective space $\CC P^{2l-1}$ by
\be  \lambda [z_0, z_1, \cdots, z_{2l-1}]=[\lambda^{a_0}z_0,  \lambda^{a_1}z_1, \cdots, \lambda^{a_{2l-1}}z_{2l-1}],\ee
such that $a_i$'s are distinct integers and $\sum_{i=0}^{2l-1} a_i$ is even. 

Let $S^1$ be the subgroup of $SL(2, \RR)$ of matrices of the form $\left(\begin{array}{cc}
                                      \cos\theta&-\sin\theta\\
                                      \sin\theta&\cos\theta
                                     \end{array}\right)$.
Then
 $$M=SL(2, \RR)\times_{S^1}\CC P^{2l-1}$$
 is a $4l$-dimensional manifold with proper and cocompact $SL(2, \RR)$ action. 
 
Consider the two-variable series
\be 
\begin{split}
&\left[\prod_{n=1}^\infty(1-q^n)^{4l}\right]\cdot\left[\prod_{n=1}^\infty\left(1+\sum_{i=1}^\infty (\lambda^2q^n)^i\right)\left(1+\sum_{i=1}^\infty (\lambda^{-2}q^n)^i \right)\right] \\
&\cdot \left[\sum_{j=0}^{2l-1}\prod_{s\neq j}\frac{1}{\left(\lambda^{\frac{|a_s-a_j|}{2}}-\lambda^{\frac{-|a_s-a_j|}{2}} \right)\prod_{n=1}^\infty(1-\lambda^{|a_s-a_j|}q^n)(1-\lambda^{-|a_s-a_j|}q^n)}\right].\\
\end{split}
\ee
As $\sum_{i=0}^{2l-1} a_i$ is even, it is not hard to see that in the above series the coefficient of each $q^n$ is a Laurent polynomial of $\lambda$ with integral coefficients. Denote the above series by 
 \be P(\cdots, \lambda^{-n}, \lambda^{-n+1}, \cdots,  \lambda^{-1}, 1, \lambda, \cdots, \lambda^{m-1}, \lambda^{m}, \cdots; q). \ee                                     
Then the Witten genus of $M$ is
\be  
\varphi_W^c(M,\tau)=P(\cdots,  -|-n|, -|-n+1|, -1, 0, -1, \cdots, -|m-1|, -|m|, \cdots; q)\in \ZZ[[q]],
\ee i.e. the $q$-series obtained by replacing each $\lambda^n$ with $-|n-1|$ in $P$. 

For the elliptic genera, one has
\be  \varphi_1^c(M,\tau)=0, \ \   \varphi_2^c(M,\tau)=0. \ee
\end{example}

The detailed computation will be given in the last section of this paper.

$\, $

Let $$ SL_2(\mathbb{Z}):= \left\{\left.\left(\begin{array}{cc}
                                      a&b\\
                                      c&d
                                     \end{array}\right)\right|a,b,c,d\in\mathbb{Z},\ ad-bc=1
                                     \right\}
                                     $$
 as usual be the modular group. Let
$$S=\left(\begin{array}{cc}
      0&-1\\
      1&0
\end{array}\right), \ \ \  T=\left(\begin{array}{cc}
      1&1\\
      0&1
\end{array}\right)$$
be the two generators of $ SL_2(\mathbb{Z})$. Their actions on
$\mathbb{H}$ are given by
$$ S:\tau\rightarrow-\frac{1}{\tau}, \ \ \ T:\tau\rightarrow\tau+1.$$
Let
$$ \Gamma_0(2)=\left\{\left.\left(\begin{array}{cc}
a&b\\
c&d
\end{array}\right)\in SL_2(\mathbb{Z})\right|c\equiv0\ \ (\rm mod \ \ 2)\right\},$$

$$ \Gamma^0(2)=\left\{\left.\left(\begin{array}{cc}
a&b\\
c&d
\end{array}\right)\in SL_2(\mathbb{Z})\right|b\equiv0\ \ (\rm mod \ \ 2)\right\}$$
be the two modular subgroups of $SL_2(\mathbb{Z})$. It is known
that the generators of $\Gamma_0(2)$ are $T,ST^2ST$, while the
generators of $\Gamma^0(2)$ are $STS,T^2STS$ (cf. \cite{Ch85}).

\begin{proposition}[Modularity] (i) If the first $G$-invariant Pontryagin class $p_1^G(TM)=0$, then $\varphi_W^c(M,\tau)$ is a modular form of weight $2k$ over $SL_2(\mathbb{Z})$. \newline
 (ii) $\varphi_1^c(M,\tau)$ is a modular form weight $2k$ over $\Gamma _{0}(2)$, while $\varphi_2^c(M,\tau)$ is a modular form weight $2k$ over $\Gamma ^{0}(2)$. Moreover, the following identity holds,
\be 
\varphi_1^c\left(M,-\frac{1}{\tau }\right)=(2\tau)^{2k}\varphi_2^c(M,\tau).
\ee

\end{proposition}
\begin{proof} 

The proof here essentially follows \cite{Liu95cmp} with performing the modular transformations on the level of forms and taking care of the $G$-invariance and the cutoff function. 

Recall that the four Jacobi theta-functions (c.f. \cite{Ch85}) defined by
infinite multiplications are

\be \theta(v,\tau)=2q^{1/8}\sin(\pi v)\prod_{j=1}^\infty[(1-q^j)(1-e^{2\pi \sqrt{-1}v}q^j)(1-e^{-2\pi
\sqrt{-1}v}q^j)], \ee
\be \theta_1(v,\tau)=2q^{1/8}\cos(\pi v)\prod_{j=1}^\infty[(1-q^j)(1+e^{2\pi \sqrt{-1}v}q^j)(1+e^{-2\pi
\sqrt{-1}v}q^j)], \ee
 \be \theta_2(v,\tau)=\prod_{j=1}^\infty[(1-q^j)(1-e^{2\pi \sqrt{-1}v}q^{j-1/2})(1-e^{-2\pi
\sqrt{-1}v}q^{j-1/2})], \ee
\be \theta_3(v,\tau)=\prod_{j=1}^\infty[(1-q^j)(1+e^{2\pi
\sqrt{-1}v}q^{j-1/2})(1+e^{-2\pi \sqrt{-1}v}q^{j-1/2})], \ee where
$q=e^{2\pi \sqrt{-1}\tau}, \tau\in \mathbb{H}$.

They are all holomorphic functions for $(v,\tau)\in \mathbb{C \times
H}$, where $\mathbb{C}$ is the complex plane and $\mathbb{H}$ is the
upper half plane.

One can express the Witten genus and the elliptic genera by using the theta functions and curvature as follows (cf. \cite{Liu95cmp}, \cite{CH} )

\be \varphi_W^c(M,\tau) =
\displaystyle\int_M c\cdot\mathrm{det}^{1\over
2}\left(\frac{R^{TM}}{4{\pi}^2}\frac{\theta'(0,\tau)}{\theta(\frac{R^{TM}}{4{\pi}^2},\tau)}
\right), \ee

\be \varphi_1^c(M,\tau) =2^{2k}\int_M c\cdot\mathrm{det}^{1\over
2}\left(\frac{R^{TM}}{4{\pi}^2}\frac{\theta'(0,\tau)}{\theta(\frac{R^{TM}}{4{\pi}^2},\tau)}
\frac{\theta_{1}(\frac{R^{TM}}{4{\pi}^2},\tau)}{\theta_{1}(0,\tau)}\right),
\ee

\be \varphi_2^c(M,\tau) =\int_M c\cdot\mathrm{det}^{1\over
2}\left(\frac{R^{TM}}{4{\pi}^2}\frac{\theta'(0,\tau)}{\theta(\frac{R^{TM}}{4{\pi}^2},\tau)}
\frac{\theta_{2}(\frac{R^{TM}}{4{\pi}^2},\tau)}{\theta_{2}(0,\tau)}\right).
\ee

The theta functions satisfy the the following
transformation laws (cf. \cite{Ch85}), 
\be 
\theta(v,\tau+1)=e^{\pi \sqrt{-1}\over 4}\theta(v,\tau),\ \ \
\theta\left(v,-{1}/{\tau}\right)={1\over\sqrt{-1}}\left({\tau\over
\sqrt{-1}}\right)^{1/2} e^{\pi\sqrt{-1}\tau v^2}\theta\left(\tau
v,\tau\right)\ ;\ee 
\be \theta_1(v,\tau+1)=e^{\pi \sqrt{-1}\over
4}\theta_1(v,\tau),\ \ \
\theta_1\left(v,-{1}/{\tau}\right)=\left({\tau\over
\sqrt{-1}}\right)^{1/2} e^{\pi\sqrt{-1}\tau v^2}\theta_2(\tau
v,\tau)\ ;\ee 
\be\theta_2(v,\tau+1)=\theta_3(v,\tau),\ \ \
\theta_2\left(v,-{1}/{\tau}\right)=\left({\tau\over
\sqrt{-1}}\right)^{1/2} e^{\pi\sqrt{-1}\tau v^2}\theta_1(\tau
v,\tau)\ ;\ee 
\be\theta_3(v,\tau+1)=\theta_2(v,\tau),\ \ \
\theta_3\left(v,-{1}/{\tau}\right)=\left({\tau\over
\sqrt{-1}}\right)^{1/2} e^{\pi\sqrt{-1}\tau v^2}\theta_3(\tau
v,\tau)\ .\ee

By applying the Chern root algorithm on the level of forms (over certain ring extension $\CC[\wedge^2_xT^*M]\subset R'$ for each $x\in M$, cf. \cite{Hu} for details) and the transformation laws of the theta functions, we have 
\be \varphi_W^c(M,-1/\tau) =\tau^{2k}\int_M c \cdot e^{\tr\left(\frac{1}{4\pi^2}(R^{TM})^2\right)} \mathrm{det}^{1\over
2}\left(\frac{R^{TM}}{4{\pi}^2}\frac{\theta'(0,\tau)}{\theta(\frac{R^{TM}}{4{\pi}^2},\tau)}
\right).\ee

If $p_1^G(TM)=0$, then $\tr\left(\frac{1}{4\pi^2}(R^{TM})^2\right)=d\omega$ for some $G$-invariant form $\omega$. Then 
\be \label{transformation}
\begin{split}
&\tau^{2k}\int_M c \cdot e^{\tr\left(\frac{1}{4\pi^2}(R^{TM})^2\right)} \mathrm{det}^{1\over
2}\left(\frac{R^{TM}}{4{\pi}^2}\frac{\theta'(0,\tau)}{\theta(\frac{R^{TM}}{4{\pi}^2},\tau)}
\right)\\
=&\tau^{2k}\int_M c\cdot  \mathrm{det}^{1\over
2}\left(\frac{R^{TM}}{4{\pi}^2}\frac{\theta'(0,\tau)}{\theta(\frac{R^{TM}}{4{\pi}^2},\tau)}\right)\\
&+ \tau^{2k} \int_M c\cdot d\left(\omega\cdot\frac{e^{\tr\left(\frac{1}{4\pi^2}(R^{TM})^2\right)-1}}{\tr\left(\frac{1}{4\pi^2}(R^{TM})^2\right)} \mathrm{det}^{1\over
2}\left(\frac{R^{TM}}{4{\pi}^2}\frac{\theta'(0,\tau)}{\theta(\frac{R^{TM}}{4{\pi}^2},\tau)}
\right)  \right).\\
\end{split}
\ee
However since $\omega\cdot \frac{e^{\tr\left(\frac{1}{4\pi^2}(R^{TM})^2\right)-1}}{\tr\left(\frac{1}{4\pi^2}(R^{TM})^2\right)} \mathrm{det}^{1\over
2}\left(\frac{R^{TM}}{4{\pi}^2}\frac{\theta'(0,\tau)}{\theta(\frac{R^{TM}}{4{\pi}^2},\tau)}\right)$ is $G$-invariant and $c$ is the cutoff function, we see that the second term in the above formula is 0. So we have 
\be \varphi_W^c(M,-1/\tau)=\tau^{2k} \varphi_W^c(M,\tau).  \ee
One also easily verifies that
\be \varphi_W^c(M, \tau+1)=\varphi_W^c(M,\tau).   \ee
Therefore (i) follows. 

Similarly, one can show that 
\be \varphi_1^c(M,-1/\tau) =(2\tau)^{2k}\varphi_2^c(M,\tau), \ \ \  \varphi_1^c(M,\tau+1)=\varphi_1^c(M,\tau). \ee
Since the generators of $\Gamma_0(2)$ are $T,ST^2ST$, while the
generators of $\Gamma^0(2)$ are $STS,T^2STS$, the modularity in (ii)  follows. 

\end{proof}

\begin{remark} In the above proposition, we use the Chern-Weil definition of $G$-invariant the Pontryagin class $p_1^G(TM)$. The total $G$-invariant Pontryagin form of a $G$-equivariant vector bundle $E$ with $G$-invariant connection $\nabla^E$ is defined as (cf. \cite{Z01})
\be p^G(E, \nabla^E)= \mathrm{det}^{1\over
2}\left(I-\left(\frac{R^E}{2\pi}\right)^2
\right). \ee
Splitting by degrees, one has
\be p^G(E, \nabla^E)=1+p_1^G(E, \nabla^E)+p_2^G(E, \nabla^E)+\cdots p_k^G(E, \nabla^E),\ee
such that $p_i^G(E, \nabla^E)\in \Omega^{4i}(M)$. The $G$-invariant Pontryagin classes $p_i^G(E)$ are the $G$-invariant cohomology classes represented by $p_i^G(E, \nabla^E)$ in the $G$-invariant de Rham cohomology $H^i(M)^G.$ Explicitly, $p_1^G(E, \nabla^E)=\tr\left(\frac{1}{4\pi^2}(R^{E})^2\right).$

\end{remark}

Following \cite{Liu95cmp},  there exists a {\it miraculous cancellation formula} on the form level,
\be \{\hat L(M)\}^{(4k)}=\sum_{j=0}^{[\frac{k}{2}]} 2^{3k-6[\frac{k}{2}]} \{\hat A(M)\mathrm{ch}(h_j(T_\CC M))\}^{(4k)},\ee 
where each $h_j(T_\CC M)\in K^0_G(M), 0\leq j\leq [\frac{k}{2}]$ is an integral linear combination of $B_i(T_\CC M)$, $0\leq i\leq j$, generalizing the celebrated Alvarez-Gaum\'e-Witten miraculous cancellation formula \cite{AGW} in dimension 12. Suppose $M$ is $G$-Spin. Multiplying $c$ on both sides, integrating over $M$ and applying Wang's index formula \cite{HW}, we get the {\it miraculous cancellation formula for proper actions}
\be\label{eq:miraculous}  \tau_*(\ind_G(\mathcal{B}_M))=\sum_{j=0}^{[\frac{k}{2}]}2^{3k-6[\frac{k}{2}]} \tau_*(\ind_G(\dirac_M^{h_j(T_\CC M)})),  \ee
where $\mathcal{B}_M$ is the signature operator.

\subsection{Witten genus and elliptic genera for proper actions: vanishing \& rigidity  }

$\, $

In this subsection, we introduce the equivariant Witten genus and the equivariant elliptic genera. We will always assume that {\it $G$ is an almost connected Lie group that  acts properly and cocompactly on a $4k$ dimensional manifold $M$ with a $G$-equivariant Spin structure.}

Define the {\it $G$-equivariant Witten genus} by
\be
\varphi_{W,G}(M,\tau) = \ind_G(\dirac_M \bigotimes\Theta(T_\CC M, \tau))\in K_0(C^*_r(G))[[q]].
\ee
That is, 
\be
\varphi_{W,G}(M,\tau) = \ind_G(\dirac_M) +  \ind_G(\dirac_M \otimes W_1(T_\CC M)) q + \cdots.
\ee
Then $\tau_*\left(\varphi_{W,G}(M)\right) = \varphi_{W}^c(M),$ where $\tau : C^*_r(G) \to \mathbb C$ denotes 
the von Neumann trace as above.

Define the {\it $G$-equivariant elliptic genera} as
\be 
\varphi_{1,G}(M,\tau) = \ind_G(\mathcal{B}_M \bigotimes\Theta_1(T_\CC M, \tau))\in K_0(C^*_r(G))[[q^{1\over2}]],
\ee
\be 
\varphi_{2,G}(M,\tau) = \ind_G(\dirac_M \bigotimes\Theta_2(T_\CC M, \tau))\in K_0(C^*_r(G))[[q^{1\over2}]].
\ee
That is, 
\be 
\varphi_{1,G}(M,\tau) =\ind_G(\mathcal{B}_M) +  \ind_G(\mathcal{B}_M \otimes A_1(T_\CC M)) q^{1\over2}+ \cdots,  
\ee
\be 
\varphi_{2,G}(M,\tau) =\ind_G(\dirac_M) +  \ind_G(\dirac_M \otimes B_1(T_\CC M)) q^{1\over2} + \cdots.  
\ee
Then $\tau_*\left(\varphi_{i,G}(M, \tau)\right) = \varphi_{i}^c(M, \tau), i=1,2$, where $\tau : C^*_r(G) \to \mathbb C$ denotes 
the von Neumann trace as above.

For the Witten genus, we conclude the following vanishing result. 

\begin{theorem}[Vanishing]\label{vanishing} Suppose $G$ is simply connected and $G/K$ has $G$-Spin structure with $K$ being a maximal compact subgroup. If $M$ is string and the $G$-action is properly non-trivial, then the $G$-equivariant Witten genus vanishes, i.e. $\varphi_{W,G}(M, \tau) =0 \in K_0(C^*_r(G))[[q]].$ 
By taking von Neumann trace,  we see that the Witten genus vanishes, i.e.
$\varphi_{W}^c(M, \tau)= \hat{A}^c(M;\bigotimes\Theta(T_\CC M, \tau))=0 \in \RR[[q]].$\newline
\end{theorem}

For the elliptic genera, we obtain the rigidity results to be stated in the following.

Consider the $q$-series:
\be \delta_1(\tau)=\frac{1}{4}+6\sum_{n=1}^{\infty}\underset{d\ odd}{\underset{d|n}{\sum}}dq^n={1\over 4}+6q+6q^2+\cdots,\ee
\be \varepsilon_1(\tau)=\frac{1}{16}+\sum_{n=1}^{\infty}\underset{d|n}{\sum}(-1)^dd^3q^n={1\over
16}-q+7q^2+\cdots,\ee
\be \delta_2(\tau)=-\frac{1}{8}-3\sum_{n=1}^{\infty}\underset{d\ odd}{\underset{d|n}{\sum}}dq^{n/2}=-{1\over 8}-3q^{1/2}-3q-\cdots,\ee
\be \varepsilon_2(\tau)=\sum_{n=1}^{\infty}\underset{n/d\ odd}{\underset{d|n}{\sum}}d^3q^{n/2}=q^{1/2}+8q+\cdots.\ee
Note that $4\delta_1(\tau), 16\varepsilon_1(\tau), 8\delta_2(\tau)$ and $\varepsilon_2(\tau)$ are all integral $q$-series.

For an integral $q$-series $a_0+a_1q^{1\over2}+\cdots+a_iq^{i\over2}+\cdots$, one can view it as an element in $R(K)[[q^{1\over2}]]$ as
$$a_0\cdot[\CC]+a_1\cdot[\CC]q^{1\over2}+\cdots+a_i\cdot[\CC]q^{i\over2}+\cdots,$$ 
where $\CC$ stands for the trivial representation of $K$.

\begin{theorem}[Rigidity] \label{rigid}  Suppose $G/K$ has $G$-Spin structure with $K$ being a maximal compact subgroup. Then in $K_0(C^*_r(G))[[q^{1\over2}]]$, $\varphi_{1,G}(M, \tau)$  is an integral linear combination of 
$$[\dirac_{G/K}\otimes(\Delta^+(\mathfrak p)\oplus\Delta^-(\mathfrak p))_\CC\otimes\Theta_1({\mathfrak p}_\CC, \tau) \otimes ((4\delta_1(\tau))^a(16\varepsilon_1(\tau))^b)];$$ 
in $K_0(C^*_r(G))[[q^{1\over2}]]$, $\varphi_{2,G}(M, \tau)$  is an integral linear combination of 
$$[\dirac_{G/K}\otimes\Theta_2({\mathfrak p}_\CC, \tau) \otimes ((8\delta_2(\tau))^a(\varepsilon_2(\tau))^b)].$$

\end{theorem}


\section{Loop Dirac induction: proof of vanishing }

In this section, we give the proof of the vanishing theorem. 

We will first need a technical lemma. Let $\eta: K\to SO(E)$ be a representation of $K$ on $E$. Complexifying $E$ gives a Hermitian metric on $E_\CC$ and an induced representation $\eta: K\to U(E_\CC)$. Using $\eta$, one can construct an Hermitian bundle $E_\eta=G\times_{K, \eta}E_\CC$ over $G/K$. 
 
Define the $(E, \eta)$-twisted Dirac induction  
\be   \DInd_{K,\, E}^G:\, R(K)\to K_0(C^*_rG)\ee
by twisting the value of the Dirac induction map to be  $\ind_G (\dirac_{G/K} \otimes E_\eta\otimes \bullet) \in K_0(C^*_r(G)).$

Let $\pi: M\to G/K$ be the projection. Let $V$ be a $K$-equivariant bundle over $N$. Then it is not hard to see that one can patch the $K$-equivariant bundle $V$ to be a $G$-equivariant bundle over $M$, which we denote by $\mathcal{V}$. 
\begin{lemma} \label{twisted} The following identity holds,
\be \DInd_{K, \, E}^G(\ind_K([\dirac_N^V]) =\ind_G([\dirac_M^{\mathcal{V}\otimes \pi^*E_\eta}]).   \ee
\end{lemma}
\begin{proof} Since $M$ is a $G$-Spin manifold, every element of $K_0^G(M)$ is represented by a twisted $G$-equivariant Dirac operator, see Section 1.
The Lemma is a consequence of Section 2 as well as Proposition 22 and Remark 20 in \cite{GMW}.
\end{proof}

Now we are ready to give the proof of the vanishing theorem.

\begin{proof}[Proof of Theorem \ref{vanishing}]

Note that after choosing a connection on the fiber bundle $M\to G/K$, we have the following decomposition
\be T_\CC M\cong T^v_\CC M\oplus \pi^*(T_\CC (G/K)), \ee
where $T^vM$ is the vertical tangent bundle. 

As the Witten bundle $\Theta$ is multiplicative and functorial, so we have
\be \Theta(T_\CC M, \tau)\cong \Theta(T^v_\CC M, \tau)\otimes \Theta(\pi^*(T_\CC (G/K)), \tau)\cong \Theta(T^v_\CC M, \tau)\otimes \pi^*(\Theta(T_\CC (G/K), \tau)).\ee
Therefore
\be  W_r(T_\CC M)\cong \oplus_{i+j=r}W_i(T^v_\CC M)\otimes \pi^*W_j(T_\CC(G/K)).    \ee

It is not hard to see that $T_\CC (G/K)$ is obtained by the adjoint representation $\rho: K\to U(\ppc)$. This $\rho$ in turn induces a series representations 
\be  \rho_{i}: K\to U(W_i(\ppc)), i=0, 1, \cdots,\ee
each corresponding to  the virtual bundle $W_i(T_\CC (G/K))$ over $G/K$.

By Lemma \ref{twisted}, we have
\be
\begin{split}
&\ind_G(\dirac_M^{W_r(T_\CC M)})\\
=&\ind_G(\dirac_M^{\oplus_{i+j=r}W_i(T^v_\CC M)\otimes \pi^*W_j(T_\CC(G/K)))} \\
=&\sum_{i+j=r}  \ind_G(\dirac_M^{W_i(T^v_\CC M)\otimes \pi^*W_j(T_\CC(G/K)})\\
=&\sum_{i+j=r} \DInd_{K,\, W_j(\ppc)}^G(\ind_K(\dirac_N^{W_i(T_\CC N)})).
\end{split}
\ee

Assembled into $q$-series, one has
$$\sum_{r=0}^\infty\ind_G(\dirac_M^{W_r(T_\CC M)})q^{r}=\sum_{r=0}^\infty( \sum_{i+j=r} \DInd_{K,\,W_j(\ppc)}^G(\ind_K(\dirac_N^{W_i(T_\CC N)})))q^{r}.$$

Now $K$ acts notrivially on $N$, and $K$ is covered by conjugate classes of maximal torus in it, maximal tori in $K$ act nontrivially on $N$. Let $T$ be a maximal torus. Since circles are dense in $T$, there exists a $S^1$ in $T$ acting nontrivially on $N$. 
If the $S^1$-action has no fixed points, then equivariant Witten genus of $N$ vanishes by the Atiyah-Bott-Segal-Singer-Lefschetz fixed point formula. Suppose that there are some fixed points. Let $EK$ be the universal $K$-principal bundle over the classifying space $BK$ of any topological group $K$. We have  the Borel fibre bundle
\[
N\stackrel{i}{\rightarrow} N\times_{K}EK\stackrel{\pi}{\rightarrow} BK.
\]
By our assupmtion $K$ is simply connected. As $K$ is the deformation retraction of $G$, $K$ is also simply connected and hence $BK$ is $3$-connected and $H^4(BK, \ZZ)\cong \ZZ$. We therefore see that there exists a commutative diagram
\begin{gather*}
\begin{aligned}
\xymatrix{
0\ar[r]  &   H^4(BK) \ar[d]  \ar[r]^{\pi^\ast \ \ \ \ \ \ \ }   & H^4(N\times_{K}EK) \ar[r]^{   \ \ \    \ \ \  \ \ \  \ \ \ i^\ast}  \ar[d] &H^4(N)    \ar@{=}[d]  \\
  & H^4(BS^1) \ar[r]^{\pi^\ast\ \ \ }  & H^4(N\times_{S^1}ES^1) \ar[r]^{   \ \ \  \ \ \    \ \ \  i^\ast} & H^4(N), \\
  }
\end{aligned}
\end{gather*}
such that the first row is exact, and maps to the second row by restricting the action to $S^1$.

When restricted to one slice, we have
\be TM|_N=TN\oplus \pp_N,\ee
where $\pp_N$ is the trivial vector bundle over $N$ with fiber $\pp$. Since $TM$ is string, $TN$ is also string. Therefore $p_1(TN)=0.$ Hence by the exactness of the first row, one has $p_1(TN)_{K}=n\cdot \pi^\ast q,$ where $q\in H^4(BK, \ZZ)\cong \ZZ$ is the generator. Restricting to the second row, we get $p_1(TN)_{S^1}=n\cdot \pi^\ast u^2,$ where $u^2$ is the generator of $H^4(BS^1, \ZZ)$. By Theorem 6 in \cite{Liu95}, we have the $S^1$-equivariant Witten genus $\sum_{i=0}^\infty\ind_{S^1}(\dirac_N^{W_i(T_\CC N)})q^i=0$, in particular, we see that the Witten genus 
$\sum_{i=0}^\infty\ind(\dirac_N^{W_i(T_\CC N)})q^i=0.$ 
On the other hand, if a circle in $T$ acts trivially on $N$, then the equivariant Witten genus is a constant, i.e. the Witten genus of $N$, which is equal to 0. Hence the circle equivariant Witten genus for circles in $T$ are all 0. Since circles are dense in $T$, we see that the $T$-equivariant Witten genus $\sum_{i=0}^\infty\ind_{T}(\dirac_N^{W_i(T_\CC N)})q^i=0$ and hence the $K$-equivariant Witten genus $\sum_{i=0}^\infty\ind_{K}(\dirac_N^{W_i(T_\CC N)})q^i=0$ and consequently $$\ind_G(\dirac_M^{W_r(T_\CC M)})=0, \forall r.$$
So $$\varphi_{W,G}(M, \tau)=\sum_{r=0}^\infty\ind_G(\dirac_M^{W_r(T_\CC M)})q^{r}=0 \in K_0(C^*_r(G))[[q]].$$
\end{proof}


\section{Loop Dirac induction: proof of rigidity} \label{sec:loop}


In this section, we study the rigidity of the equivariant elliptic genera. To understand the statement of Theorem \ref{rigid} more clearly, let us look at the $q$-series with coefficients in $K_0(C^*_r(G))$:
\be F_1(\tau):=[\dirac_{G/K}\otimes(\Delta^+(\mathfrak p)\oplus\Delta^-(\mathfrak p))_\CC\otimes\Theta_1({\mathfrak p}_\CC, \tau)]\in K_0(C^*_r(G),\ee
\be F_2(\tau):=[\dirac_{G/K}\otimes\Theta_2({\mathfrak p}_\CC, \tau)]\in K_0(C^*_r(G),\ee
which appear in Theorem \ref{rigid}. Explicitly, the first several terms are
\be
\begin{split} &\Theta_1({\mathfrak p}_\CC, \tau)\\
=&\bigotimes_{n=1}^\infty
S_{q^n}(\ppc)\Lambda_{-q^n}(\CC^{d})\otimes \bigotimes_{m=1}^\infty \Lambda_{q^m}(\ppc
)S_{-q^m}(\CC^d)\\
=&(\CC+\ppc q+S^2\ppc q^2+\cdots)(\CC+\ppc q^2+\cdots)(\CC-\CC^{d}q+\Lambda^2\CC^{d}q^2+\cdots)(\CC-\CC^{d}q^2+\cdots)\\
&(\CC+\ppc q+\Lambda^2\ppc q^2+\cdots)(\CC+\ppc q^2+\cdots)(\CC-\CC^{d}q+S^2\CC^{d}q^2+\cdots)(\CC-\CC^{d}q^2+\cdots)\\
=&(\CC+2\ppc q+(\ppc\otimes \ppc+S^2\ppc+\Lambda^2\ppc
M)q^2+\cdots)(\CC+2\ppc q^2+\cdots)\\
&(\CC-2\CC^{d}q+(\CC^{d}\otimes \CC^{d}+S^2\CC^{d}
+\Lambda^2\CC^{d})q^2+\cdots)(\CC-2\CC^{d}q^2+\cdots)\\
=&(\CC+2\ppc q+2(\ppc+ \ppc\otimes \ppc)q^2+\cdots)(\CC-2\CC^{d}q+2(\CC^{d}\otimes\CC^{d}-\CC^{d})q^2+\cdots)\\
=&\CC+2(\ppc-\CC^{d})q+2[\ppc\otimes \ppc-(2d-1)\ppc
+\CC^{d(d-1)}]q^2+\cdots,\\
\end{split} \ee where the ``$\cdots$" are the terms involving
$q^{j}$'s with $j \geq 3$. So
\be 
\begin{split}
&F_1(\tau)\\
=&[\dirac_{G/K}\otimes(\Delta^+(\mathfrak p)\oplus\Delta^-(\mathfrak p))_\CC]+2[\dirac_{G/K}\otimes(\Delta^+(\mathfrak p)\oplus\Delta^-(\mathfrak p))_\CC\otimes({\mathfrak p}_\CC-\CC^d)]q\\
&+2[\dirac_{G/K}\otimes(\Delta^+(\mathfrak p)\oplus\Delta^-(\mathfrak p))_\CC\otimes(\ppc\otimes \ppc-(2d-1)\ppc
+\CC^{d(d-1)})]q^2+\cdots,\\
\end{split}
\ee
where the ``$\cdots$" are the terms involving
$q^{j}$'s with $j \geq 3$.

\be
\begin{split} &\Theta_2({\mathfrak p}_\CC, \tau)\\
=&\bigotimes_{n=1}^\infty
S_{q^n}(\ppc)\Lambda_{-q^n}(\CC^{d})\otimes \bigotimes_{m=1}^\infty \Lambda_{-q^{m-\frac{1}{2}}}(\ppc
)S_{q^{m-\frac{1}{2}}}(\CC^d)\\
=&(\CC+\ppc q+\cdots)(\CC-\CC^{d}q+\cdots)(\CC-\ppc q^{\frac{1}{2}}+\Lambda^2\ppc q+\cdots)(\CC+\CC^{d}q^{\frac{1}{2}}+\CC^{\frac{d(d+1)}{2}}q+\cdots)\\
=&\CC-(\ppc-\CC^d)q^{\frac{1}{2}}+(\Lambda^2\ppc-(d-1)\ppc+\CC^{\frac{d(d-1)}{2}})q+\cdots
\end{split} \ee where the ``$\cdots$" are the terms involving
$q^{j}$'s with $j \geq \frac{3}{2}$. So 
\be 
F_2(\tau)=[\dirac_{G/K}]-[\dirac_{G/K}\otimes({\mathfrak p}_\CC-\CC^d)]q^{1\over 2}+[\dirac_{G/K}\otimes(\Lambda^2\ppc-(d-1)\ppc+\CC^{\frac{d(d-1)}{2}})]q+\cdots,
\ee
where the ``$\cdots$" are the terms involving
$q^{j}$'s with $j \geq \frac{3}{2}$.

\begin{proof}[Proof of Theorem \ref{rigid}] 

Suppose the dimension of $N$ is $4m$. Denote the elliptic genera of $N$ by $P_1(N, \tau)$ and $P_2(N, \tau)$.  It is known that $P_1(N, \tau)$ is an integral modular form of weight $2m$ over $\Gamma_0(2)$; while $P_2(N, \tau)$ is an integral modular
form of weight $2m$ over $\Gamma^0(2)$. Moreover, the following
identity holds, \be P_1(N, -1/\tau)={(2\tau)}^{2m}P_2(N, \tau).\ee

If $\Gamma$ is a modular subgroup, let
$M(\Gamma)$ denote the ring of modular forms
over $\Gamma$ with integral Fourier coefficients.

\begin{lemma} [\protect cf. \cite{Liu95}] \label{modular} One has that $\delta_1(\tau)\ (resp.\ \varepsilon_1(\tau) ) $
is a modular form of weight $2 \ (resp.\ 4)$ over $\Gamma_0(2)$,
$\delta_2(\tau) \ (resp.\ \varepsilon_2(\tau))$ is a modular form
of weight $2\ (resp.\ 4)$ over $\Gamma^0(2)$ and moreover
$M(\Gamma^0(2))=\mathbb{Z}[8\delta_2(\tau), \varepsilon_2(\tau)]$. Moreover, we have
transformation laws \be
\delta_2\left(-\frac{1}{\tau}\right)=\tau^2\delta_1(\tau),\ \ \ \ \
\ \ \ \ \
\varepsilon_2(\tau)\left(-\frac{1}{\tau}\right)=\tau^4\varepsilon_1(\tau).\ee
\end{lemma}

One can apply Lemma \ref{modular} to $P_2(N, \tau)$ to get that 
\be \label{p2}
P_2(N, \tau)
 =h_0(8\delta_2(\tau))^m+h_1(8\delta_2(\tau))^{m-2}\varepsilon_2(\tau)
+\cdots+h_{[\frac{m}{2}]}(8\delta_2(\tau))^{\bar m}\varepsilon_2(\tau)^{[\frac{m}{2}]} ,\ee where 
$\bar m=0$ if $m$ is even and $\bar m=1$ if $m$ is odd, and each $h_r(T_\CC M)$, $0\leq r\leq [\frac{m}{2}] $, is an integer. Actually they are all indices of certain twisted Dirac operators on $N$.

By the above two relations, one has \be \label{p1}
\begin{split}P_1(N, \tau)=&2^{2m}\frac{1}{\tau^{2m}}P_2(N, -1/\tau)\\
=&2^{2m}\frac{1}{\tau^{2m}}\Big[h_0\big(8\delta_2(-1/\tau)\big)^{m} +h_1\big(8\delta_2(-1/\tau)\big)^{m-2}\varepsilon_2(-1/\tau)
+\cdots\\
&+h_{[\frac{m}{2}]}\big(8\delta_2(-1/\tau)\big)^{\bar m}\big(\varepsilon_2(-1/\tau)\big)^{[\frac{m}{2}]}\Big]\\
=&2^{2m}\left[h_0(8\delta_1(\tau))^{m}+h_1(8\delta_1(\tau))^{m-2}\varepsilon_1(\tau) +\cdots+h_{[\frac{m}{2}]}(8\delta_1(\tau))^{\bar m}\varepsilon_1(\tau)^{[\frac{m}{2}]}\right]\\
=&2^{3m}h_0(4\delta_1(\tau))^{m}+2^{3m-6}h_1(4\delta_1(\tau))^{m-2}(16\varepsilon_1(\tau))+\cdots+2^{3m-6[\frac{m}{2}]}h_{[\frac{m}{2}]}(4\delta_1(\tau))^{\bar m}(16\varepsilon_1(\tau))^{[\frac{m}{2}]}.
\end{split}\ee

By the similar steps in the proof of vanishing theorem \ref{vanishing}, we have
\be \ind_G(\mathcal{B}_M^{A_r(T_\CC M)})=\sum_{i+j=r} \DInd_{K, \, A_j(\ppc)\otimes (\Delta^+(\mathfrak p)\oplus\Delta^-(\mathfrak p))_\CC}^G(\ind_K(\mathcal{B}_N^{A_i(T_\CC N)})).
\ee
So we have
\be 
\begin{split}
&\varphi_{1,G}(M, \tau)\\
=&\sum_{r=0}^\infty\ind_G(\mathcal{B}_M^{W_r(T_\CC M)})q^{\frac{r}{2}}\\
=&\DInd_{K,\, (\Delta^+(\mathfrak p)\oplus\Delta^-(\mathfrak p))_\CC\otimes \Theta_1(\ppc, \tau)}^G\left(\sum_{i=0}^\infty \ind_K(\mathcal{B}_N^{A_i(T_\CC N)})\right).\\
\end{split}
\ee
By the Witten-Bott-Taubes-Liu's rigidity theorem, one sees that 
$$\sum_{i=0}^\infty \ind_K(\mathcal{B}_N^{A_i(T_\CC N)})=P_1(N, \tau)\cdot[\CC]. $$
Hence we have
\be \label{rigid1}
\begin{split}
&\varphi_{1,G}(M, \tau)\\
=&2^{3m}h_0\DInd_{K,\, (\Delta^+(\mathfrak p)\oplus\Delta^-(\mathfrak p))_\CC\otimes \Theta_1(\ppc, \tau)}^G((4\delta_1(\tau))^{m})\\
&+2^{3m-6}h_1\DInd_{K, \,(\Delta^+(\mathfrak p)\oplus\Delta^-(\mathfrak p))_\CC\otimes \Theta_1(\ppc, \tau)}^G((4\delta_1(\tau))^{m-2}(16\varepsilon_1(\tau)))\\
&+\cdots+2^{3m-6[\frac{m}{2}]}h_{[\frac{m}{2}]}\DInd_{K,\, (\Delta^+(\mathfrak p)\oplus\Delta^-(\mathfrak p))_\CC\otimes \Theta_1(\ppc, \tau)}^G((4\delta_1(\tau))^{\bar m}(16\varepsilon_1(\tau))^{[\frac{m}{2}]}).\\
\end{split}
\ee

Similarly, we can show that

\be \label{rigid2}
\begin{split}
&\varphi_{2,G}(M, \tau)\\
=&h_0\DInd_{K,\, \Theta_2(\ppc, \tau)}^G((8\delta_2(\tau))^m)+h_1\DInd_{K,\, \Theta_2(\ppc, \tau)}^G((8\delta_2(\tau))^m)((8\delta_2(\tau))^{m-2}\varepsilon_2(\tau))\\
&+\cdots+h_{[\frac{m}{2}]}\DInd_{K, \,\Theta_2(\ppc, \tau)}^G((8\delta_2(\tau))^m)((8\delta_2(\tau))^{\bar m}\varepsilon_2(\tau)^{[\frac{m}{2}]}). \\
\end{split}
\ee
So the desired result follows.

\end{proof}


\section{Some examples}

In this section, we give the computation of the Witten genus and elliptic genera in Example \ref{example}. 

Now $S^1$ act on the complex projective space $\CC P^{2l-1}$ by
\be  \lambda [z_0, z_1, \cdots, z_{2l-1}]=[\lambda^{a_0}z_0,  \lambda^{a_1}z_1, \cdots, \lambda^{a_{2l-1}}z_{2l-1}],\ee
such that $a_i$'s are distinct integers and $\sum_{i=0}^{2l-1} a_i$ is even. 

Since $a_i$'s are distinct from each, we see that this action has $2l$ fixed points 
$$[1, 0, \cdots, 0], [0,1, 0, \cdots, 0], \cdots, [0,0, \cdots, 1]. $$ 
At the $j$-th fixed point, $0\leq j\leq 2l-1$, one has
$$\sum_{s\neq j}|a_k-a_j|\equiv \sum_{i=0}^{2l-1} a_i\equiv 0\mod 2. $$
By \cite{AH70}, we know that the circle action preserves the Spin structure of $\CC P^{2l-1}.$

Therefore one can apply the Atiyah-Bott-Segal-Singer Lefschetz fixed-point formula to find that the Lefschetz number of the $S^1$-equivariant Witten genus is
\be  
\begin{split}
&\Ind( \dirac_{\CC P^{2l-1}} \bigotimes\Theta(T_\CC \CC P^{2l-1}, \tau)))(\lambda)\\
=&\prod_{n=1}^\infty(1-q^n)^{4l-2}\sum_{j=0}^{2l-1}\prod_{s\neq j}\frac{1}{\left(\lambda^{\frac{|a_s-a_j|}{2}}-\lambda^{\frac{-|a_s-a_j|}{2}} \right)\prod_{n=1}^\infty(1-\lambda^{|a_s-a_j|}q^n)(1-\lambda^{-|a_s-a_j|}q^n)}, \ \lambda\in S^1. 
\end{split}
 \ee
 
In the $q$-expansion,   
\be \Ind( \dirac_{\CC P^{2l-1}} \bigotimes\Theta(T_\CC \CC P^{2l-1}, \tau)))(\lambda)=\sum_{i=0}^\infty a_i(\lambda)q^i, \ee
each coefficient $a_i(\lambda)$ is an integral Laurent polynomial of $\lambda$. 

Let $\CC[n]$ denote the representation of $S^1$ on $\CC$ by $\lambda \mapsto \lambda^n\cdot$. Clearly 
$$\CC[n]\otimes \CC[m]\cong \CC[n+m]. $$

Then  in $R[S^1][[q]]$, we must have
\be \Ind_{S^1}( \dirac_{\CC P^{2l-1}} \bigotimes\Theta(T_\CC \CC P^{2l-1}, \tau)))=\sum_{i=0}^\infty a_i(\CC[1])q^i.\ee

In $SL(2, \RR)$, consider the circle subgroup 
$$\left\{\left(\begin{array}{cc}\cos\theta&-\sin\theta\\
                                      \sin\theta&\cos\theta
                                     \end{array}\right)\right\}.$$
The Lie algebra of this subgroup has generator $$X=\left(\begin{array}{cc}
                                     0&-1\\
                                      1&0
                                     \end{array}\right).$$                                     
Let $\pp$ be the orthogonal complement of $\RR\cdot X$ in $sl(2, \RR)$.  It is not hard to see that $\ppc$ is generated by 
$$E_{+}=\left(\begin{array}{cc}
                                     1&i\\
                                      i&-1
                                     \end{array}\right), \ \ \ \ E_{-}=\left(\begin{array}{cc}
                                     1&-i\\
                                      -i&-1
                                     \end{array}\right).$$       
Simple computation shows that 
$$[X, E_{\pm}]=\pm2iE_{\pm}.$$
This shows that the adjoint represntation of $S^1$ on $\ppc$ is isomorphic to $\CC[2]\oplus\CC[-2].$ 
Therefore in $\RR(S^1)[[q]]$, we have
\be  
\begin{split}
&\Theta(\ppc, \tau)\\
\cong &\bigotimes_{n=1}^\infty S_{q^n}(\CC[2])\otimes \bigotimes_{n=1}^\infty S_{q^n}(\CC[-2])\otimes  \prod_{n=1}^\infty(\CC[0]-q^n)^{\otimes^2}\\
\cong & \bigotimes_{n=1}^\infty\left(\left(\oplus_{i=0}^\infty \CC[2i]q^{ni}\right)\otimes\left(\oplus_{i=0}^\infty \CC[-2i]q^{ni}\right)\right)\otimes  \bigotimes_{n=1}^\infty(\CC[0]-q^n)^{\otimes^2}.  
\end{split}
\ee

By the proof of Theorem \ref{vanishing}, we see that the $SL(2, \RR)$-equivariant Witten genus 
\be 
\begin{split}
&\varphi_{W,SL(2, \RR)}(M, \tau)\\
=& \DInd_{S^1, \Theta(\ppc, \tau)}^{SL(2, \RR)} (\Ind_{S^1}( \dirac_{\CC P^{2l-1}} \bigotimes\Theta(T_\CC \CC P^{2l-1}, \tau))))\\
=& \DInd_{S^1}^{SL(2, \RR)} (\Theta(\ppc, \tau)\otimes \sum_{i=0}^\infty a_i(\CC[1])q^i)\\
=& \DInd_{S^1}^{SL(2, \RR)} (\bigotimes_{n=1}^\infty\left(\left(\oplus_{i=0}^\infty \CC[2i]q^{ni}\right)\otimes\left(\oplus_{i=0}^\infty \CC[-2i]q^{ni}\right)\right)\otimes  \bigotimes_{n=1}^\infty(\CC[0]-q^n)^{\otimes^2}\otimes \sum_{i=0}^\infty a_i(\CC[1])q^i).
\end{split} \ee

Now we need the following theorem, which is extracted from Proposition 50  from \cite{GMW}.
\begin{theorem}
Let $G$ be a connected, semisimple Lie group with finite centre and $K$ a maximal compact subgroup, such that $G/K$ is Spin and even dimensional.
Let $\mu$ be the highest weight of an irreducible representation $V_\mu$ of $K$. Suppose $\mu+\rho_c$ is the Harish-Chandra parameter of a discrete series representation of $(H,\pi)$ of $G$. Then the formal degree of $(H,\pi)$ is given by
		\begin{align*}
		d_H &= (-1)^{d/2}\,\tau_*([\DInd([V_\mu])])=\prod_{\alpha\in\Phi^+}\frac{(\mu+\rho_c,\alpha)}{(\rho,\alpha)},
		\end{align*}
		where $\tau$ is the von Neumann trace of $G$ with both sides vanishing when the class $\DInd[V_{\mu}]\in K_0(C^*_r(G))$ is not given by a discrete series representation. We have a well-defined commutative diagram:
		\begin{equation*}\label{tracediagram}
		\begin{tikzcd}
		K_0(C_r^*(G))\arrow{rd}{\tau_*} & \ \\ 
		\ & \mathbb{R},\\
		R(K)\arrow{uu}{\DInd}\arrow{ru}{\Pi_K} & \
		\end{tikzcd}
		\end{equation*}
		where $\Pi_K([V_\mu]):=(-1)^{d/2}\,\prod_{\alpha\in\Phi^+}\frac{(\mu+\rho_c,\alpha)}{(\rho,\alpha)}.$ Here the Haar measure on $G$ is normalised by
		$$\textnormal{vol }K = \textnormal{vol }M_1/K_1 = 1,$$
		where $M_1$ is a maximal compact subgroup of the universal complexification $G^\mathbb{C}$ of $G$ and $K_1<G_1$ a maximal compact subgroup of a real form $G_1$ of $G^\mathbb{C}$ (see \cite{AtiyahSchmid} for more details).
\end{theorem}

But we know that 
\be\Pi_{S^1}(\CC[n])=-|n|. \ee
So we have 
\be
\begin{split}
&\varphi_{W}^c(M, \tau)\\
=&\tau_* (\varphi_{W,SL(2, \RR)}(M, \tau))\\
=&\Pi_{S^1}\left(\bigotimes_{n=1}^\infty\left(\left(\oplus_{i=0}^\infty \CC[2i]q^{ni}\right)\otimes\left(\oplus_{i=0}^\infty \CC[-2i]q^{ni}\right)\right)\otimes  \bigotimes_{n=1}^\infty(\CC[0]-q^n)^{\otimes^2}\otimes \sum_{i=0}^\infty a_i(\CC[1])q^i\right)\\
=&P(\cdots,  -|-n|, -|-n+1|, -1, 0, -1, \cdots, -|m-1|, -|m|, \cdots; q).
\end{split}
 \ee
The desired formula follows.  We summarize the above computation in the following theorem.\\

\begin{theorem}\label{thm:computation}
Let $M= SL(2, \RR) \times_{S^1} \CC P^{2l-1}$ as in Example \ref{example}. Then the $SL(2, \RR)$-equivariant Witten genus
\be\nonumber
\begin{split}
&\varphi_{W,SL(2, \RR)}(M, \tau)\\
=& \DInd_{S^1}^{SL(2, \RR)} (\bigotimes_{n=1}^\infty\left(\left(\oplus_{i=0}^\infty \CC[2i]q^{ni}\right)\otimes\left(\oplus_{i=0}^\infty \CC[-2i]q^{ni}\right)\right)\otimes  \bigotimes_{n=1}^\infty(\CC[0]-q^n)^{\otimes^2}\otimes \sum_{i=0}^\infty a_i(\CC[1])q^i).
\end{split} \ee
and the Witten genus
 \be\nonumber
\begin{split}
&\varphi_{W}^c(M, \tau)= \tau_* (\varphi_{W,SL(2, \RR)}(M, \tau))\\
=&\Pi_{S^1}\left(\bigotimes_{n=1}^\infty\left(\left(\oplus_{i=0}^\infty \CC[2i]q^{ni}\right)\otimes\left(\oplus_{i=0}^\infty \CC[-2i]q^{ni}\right)\right)\otimes  \bigotimes_{n=1}^\infty(\CC[0]-q^n)^{\otimes^2}\otimes \sum_{i=0}^\infty a_i(\CC[1])q^i\right)\\
=&P(\cdots,  -|-n|, -|-n+1|, -1, 0, -1, \cdots, -|m-1|, -|m|, \cdots; q).
\end{split}
\ee
where $P$ is the two-variable series shown in Example \ref{example}. \\
\end{theorem}

For the elliptic genera, by the Witten-Bott-Taubes-Liu's rigidity theorem, we see that
the Lefschetz number of the $S^1$-equivariant elliptic genera 
\be  
\Ind( \mathcal{B}_{\CC P^{2l-1}} \bigotimes\Theta_1(T_\CC \CC P^{2l-1}, \tau)))(\lambda), \ \ \ \ \ \ \Ind( \dirac_{\CC P^{2l-1}} \bigotimes\Theta_2(T_\CC \CC P^{2l-1}, \tau)))(\lambda)\ee
are both independent of $\lambda$. However when $\lambda=1$, they are just the elliptic genera of $\CC P^{2l-1}$, which are 0 since $\CC P^{2l-1}$ has dimension $4l-2$. Therefore 
$$\Ind( \mathcal{B}_{\CC P^{2l-1}} \bigotimes\Theta_1(T_\CC \CC P^{2l-1}, \tau)))(\lambda)$$ and 
$$\Ind( \dirac_{\CC P^{2l-1}} \bigotimes\Theta_2(T_\CC \CC P^{2l-1}, \tau)))(\lambda)$$ are constantly 0. So the $S^1$-equivariant elliptic genera
$$\Ind_{S^1}( \mathcal{B}_{\CC P^{2l-1}} \bigotimes\Theta_1(T_\CC \CC P^{2l-1}, \tau))) $$
and 
$$\Ind_{S^1}( \dirac_{\CC P^{2l-1}} \bigotimes\Theta_2(T_\CC \CC P^{2l-1}, \tau)))$$
are both 0 in $R(S^1)[[q^{1/2}]]$. From the proof of Theorem {rigidity}, we see that 
$$\varphi_{1, SL(2, \RR)}(M, \tau), \ \ \varphi_{2, SL(2, \RR)}(M, \tau)$$ are both 0 in 
$K_0(C^*_r(SL(2, \RR)))$. Taking von Neumann trace, we get
\be  \varphi_1^c(M,\tau)=0, \ \   \varphi_2^c(M,\tau)=0. \ee

\section*{Appendix}

In this appendix, we study equivariant Witten genus and equivariant elliptic genera in the situation that $G/K$ is not $G$-Spin. 

We first outline the quantisation commutes with loop Dirac induction diagrams when $G/K$ is even dimensional and not $G$-Spin. 
Assume now that the homomorphism $K\to
\mathrm{SO}(\mathfrak p)$ does not lift to $\mathrm{Spin}(\mathfrak p)$. 
 Let $\pi:\widetilde G \to G$ be the Spin double cover and $\widetilde K=\pi^{-1}(K)$
be the induced double cover $\widetilde K$ of $K$, such that  the homomorphism $K\to
\mathrm{SO}(\mathfrak p)$ lifts to a homomorphism $\widetilde K\to
\mathrm{Spin}(\mathfrak p)$. Let $\Delta(\mathfrak p)=\Delta^+(\mathfrak p)\oplus\Delta^-(\mathfrak p)$ denote the spinor 
representation of $\mathrm{Spin}(\mathfrak p)$ (cf. Definition 5.11 in \cite{LM}).
Then $\Delta(\mathfrak p) \in R(\widetilde K)$ and so is a projective representation of $K$
such that if  $u \in \ker(\pi)$
is the non-trivial element, then $u$ acts non-trivially on $\Delta(\mathfrak p)$. 
In particular, the manifold $G/K$ does not have a $G$-equivariant Spin structure.

Define $$R(K)^{spin} = \left\{ V\in R(\widetilde K): u\ \mathrm{acts\,\, nontrivially\,\, on}\ V\right\}.$$ If $V\in R(K)^{spin}$, then
$\Delta(\mathfrak p)\otimes V$ is a $K$-representation. Note that $R(\widetilde K) = R(K)^{spin} \oplus R(K)$, and denote by
$p_-: R(\widetilde K) \to R(K)^{spin}$  the projection. 

Let $\{X_j\}$ be an orthonormal basis of $\mathfrak p$ and $c$ the Clifford action by $\mathfrak{p}$ on $\Delta(\mathfrak p)$. The
Dirac operator
$$
D = \sum X_j \otimes c(X_j)
$$
is projectively $G$-invariant. Let $ V\in R(K)^{spin}$, then the twisted Dirac operator
$$
D_V = D \otimes I_V
$$
is a $G$-invariant operator. The {\em twisted Dirac induction} map
$$
\DsInd_K^G : R(K)^{spin} \longrightarrow K_0(C_r^*(G))
$$
defined as $\DsInd_K^G(V) = {\rm index}_G(D_V)$, 
is known to be an isomorphism when $G$ is an almost connected Lie group (cf. \cite{CEN}).

Clearly, $M= {\widetilde G} \times_{\widetilde K} N$.
Then by the {\em quantisation commutes with induction} diagram \cite{HM16, GMW}, one has
\be \label{qi2}
\xymatrix{
K_{0}^{\widetilde G}(M) \ar[rrrr]^-{ \ind_{\widetilde G}} & & & & K_{0}(C^*_{r}(\widetilde G)) \\
K_{0}^{\widetilde K}(N) \ar[u]^-{{ K}-{\rm Ind}_{\widetilde K}^{\widetilde G}} \ar[rrrr]^-{\ind_{\widetilde K}} & & &
 & R(\widetilde K)  \ar[u]_-{\DInd_{\widetilde K}^{\widetilde G}
 } 
}.
\ee
Choosing a left invariant measure on $\widetilde G$, it induces a left invariant measure on $G$. 
Consider $f \in L^1(\widetilde G)$. One has $f(g) = \frac{1}{2} \left(f(ug) + f(g)\right) +   \frac{1}{2} \left(-f(ug) + f(g)\right)$.
Then $f_1(g) =  \frac{1}{2} \left(f(ug) + f(g)\right) $ satisfies $f_1(ug) = f_1(g)$ and 
$f_2(g)=  \frac{1}{2} \left(-f(ug) + f(g)\right)$ satisfies $f_2(ug) = -f_2(g)$, so that we get the decomposition
$$
L^1(\widetilde G) = L^1(G)  \oplus L^1(G, \chi) 
$$
where $\chi$ is the nontrivial unitary character of $\ker(\pi)=\ZZ_2$ such that $\chi(u)=-1$.

Recall that the reduced $C^*$-algebra $C^*_r(\widetilde G)$ is the operator norm closure of 
$L^1(\widetilde G)$ acting on $L^2(\widetilde G)$ by convolution, and similarly for $G$, so we
deduce that 
$$
C^*_r(\widetilde G) = C^*_r(G)  \oplus C^*_r(G, \chi).
$$
Let $q_{+}: C^*_r(\widetilde G)  \to C^*_r(G)$ denote the projection, and we denote by the 
same symbol the projection in K-theory,
$$
q_+: K_0(C^*_r(\widetilde G))\to K_0(C^*_r(G)).$$


Since the map $\DInd_{\widetilde K}^{\widetilde G}$  is constructed by tensoring with the spinor bundle $\tilde G/\tilde K\times_{\tilde K} \Delta(\mathfrak p)$, on which $u$ acts nontrivially,  the projection $p_-$  is mapped to the projection $q_+$.  Therefore we get the following commutative diagram, 
\be \label{qi3}
\xymatrix{
K_{0}(C^*_{r}(\widetilde G)) \ar[rrr]^-{q_+}& & &   K_0(C^*_r(G)) \\
 R(\widetilde K) \ar[u]^-{\DInd_{\widetilde K}^{\widetilde G} } 
 \ar[rrr]^-{p_-}   & & &  R(K)^{spin}\ar[u]_-{\DsInd_{K}^{ G}
 } 
}
\ee
which in turn gives rise to the commutative diagram,
\be \label{qi4}
\xymatrix{
K_{0}^{\widetilde G}(M) \ar[rrrr]^-{q_+( \ind_{\widetilde G})} & & & & K_{0}(C^*_{r}(G)) \\
K_{0}^{\widetilde K}(N) \ar[u]^-{{ K}-{\rm Ind}_{\widetilde K}^{\widetilde G}} \ar[rrrr]^-{p_-(\ind_{\widetilde K})} & & &
 & R(K)^{spin} \ar[u]_-{\DsInd_{K}^{G}
 } 
}.
\ee

Using this, we can easily  modify our earlier results
 to establish the commutativity
of the following diagram,
\be \label{commW+}
\xymatrix{
K_{0}^{\widetilde G}(M)\ar[rrrr]^-{ q_+(\ind_{\widetilde G}(\bullet \,\,\, \otimes \Theta(T_\CC M,\tau)))} & & & & K_{0}(C^*_{r}G) [[q]]\\
K_{0}^{\widetilde K}(N)  \ar[u]^-{{ K}-{\rm Ind}_{\widetilde K}^{\widetilde G}}\ar[rrrr]^-{p_-(\ind_{\widetilde K}(\bullet \,\,\, \otimes \Theta(T_\CC N, \tau)))} & & &
 & R(K)^{spin} [[q]] \ar[u]_-{\DsInd_{LK}^{LG}
 }
 },
\ee
where $\DsInd_{LK}^{LG}(\bullet) = \ind_G(\dirac_{G/K} \otimes \Theta(\ppc, \tau)\otimes\,\,\,\bullet)$ is a loop version of Dirac induction. 

Similarly, we can also establish the commutativity
of the following diagram,
\be  \label{commE2+}
\xymatrix{
K_{0}^{\widetilde G}(M)\ar[rrrr]^-{ q_+(\ind_{\widetilde G}(\bullet \,\,\, \otimes\Theta_2(T_\CC M,\tau)))} & & & & K_{0}(C^*_{r}G) [[q^{1\over2}]]\\
K_{0}^{\widetilde K}(N)  \ar[u]^-{{ K}-{\rm Ind}_{\widetilde K}^{\widetilde G}} \ar[rrrr]^-{p_-(\ind_{\widetilde K}(\bullet \,\,\, \otimes \Theta_2(T_\CC N, \tau)))} & & &
 & R(K)^{spin} [[q^{1\over2}]] \ar[u]_-{\DsIndb_{LK}^{LG}
 } 
},
\ee
where $\DsIndb_{LK}^{LG}$ is a loop version of Dirac induction given explicitly by 
$$\ind_G(\dirac_{G/K} \otimes \Theta_2(\ppc, \tau)\otimes \,\,\,\bullet).$$

However a subtle point here is that there is not a ``spin" version of $\DInda_{LK}^{LG}$ like $\DsIndb_{LK}^{LG}$, since 
in the Dirac induction given by
$$\ind_G(\dirac_{G/K}\otimes(\Delta^+(\mathfrak p)\oplus\Delta^-(\mathfrak p))_\CC\otimes\Theta_1({\mathfrak p}_\CC, \tau)\otimes\,\,\,\bullet),$$ 
the operator $\dirac_{G/K}\otimes(\Delta^+(\mathfrak p)\oplus\Delta^-(\mathfrak p))_\CC$ is just the signature operator, an honest operator rather than a projective operator on $G/K$ and therefore the $\bullet$ should be in $R(K)[[q^{1\over2}]]$ rather than in $R(K)^{spin} [[q^{1\over2}]]$.  So we will still use original commutative diagram,
\be \label{commE1+}
\xymatrix{
K_{0}^G(M) \ar[rrrrrr]^-{ \ind_G(\bullet \,\,\, \otimes \Delta^+(TM)\oplus\Delta^-(TM))_\CC\otimes\Theta_1(T_\CC M,\tau))} & & & & & & K_{0}(C^*_{r}G) [[q^{1\over2}]]\\
K_{0}^K(N) \ar[u]^-{K-{\rm Ind}_K^G} \ar[rrrrrr]^-{\ind_K(\bullet \,\,\, \otimes \Delta^+(TN)\oplus\Delta^-(TN))_\CC\otimes \Theta_1(T_\CC N, \tau))} & & & & &
 & R(K) [[q^{1\over2}]] \ar[u]_-{\DInda_{LK}^{LG}
 } 
}.
\ee


Let $M$ be Spin but not $G$-Spin. Define the {\it equivariant Witten genus} by
\be
\varphi_{W,G}(M,\tau) =q_+\left( \ind_{\widetilde G}(\dirac_M \bigotimes\Theta(T_\CC M, \tau))\right)\in K_0(C^*_r(G))[[q]]
\ee
and the 
 {\it equivariant elliptic genera} by
\be 
\varphi_{1,G}(M,\tau) =\ind_{G}(\mathcal{B}_M \bigotimes\Theta_1(T_\CC M, \tau))\in K_0(C^*_r(G))[[q^{1\over2}]],
\ee
\be 
\varphi_{2,G}(M,\tau) =q_+\left( \ind_{\widetilde G}(\dirac_M \bigotimes\Theta_2(T_\CC M, \tau))\right)\in K_0(C^*_r(G))[[q^{1\over2}]].
\ee

If $\pi_1(G)=\Z_2$, then $\widetilde G$ is simply connected and therefore $\widetilde K$ is simply connected. Applying the commutative diagram (\ref{commW+}) similarly in the proof of Theorem \ref{vanishing}, we have
\begin{theorem}[Vanishing]\label{vanishing+} If $M$ is string, $\pi_1(G)=\Z_2$ and the $G$-action is properly non-trivial, then the $G$-equivariant Witten genus vanishes, i.e. $\varphi_{W,G}(M, \tau) =0 \in K_0(C^*_r(G))[[q]].$ 
\end{theorem}

Applying the commutative diagrams (\ref{commE1+}), (\ref{commE2+}) similarly in the proof of Theorem \ref{rigid} and noticing that $p_-(\CC)=0$, where $\CC$ is a trivial representation of $\widetilde K$, we have
\begin{theorem}[Rigidity] \label{rigid+}  In $K_0(C^*_r(G))[[q^{1\over2}]]$, $\varphi_{1,G}(M, \tau)$  is an integral linear combination of 
$$\DInda_{LK}^{LG}\left((4\delta_1(\tau))^a(16\varepsilon_1(\tau))^b\right);$$ 
and $\varphi_{2,G}(M, \tau)=0\in K_0(C^*_r(G))[[q^{1\over2}]]$. 
\end{theorem}


\begin{thebibliography}{10}

 \bibitem{Abels} {H.\ Abels},  {\em Parallelizability of proper actions, global K-slices and maximal compact subgroups}, 
{ Math.\ Ann.}\ {\bf 212} (1974), 1--19.




 \bibitem{AGW} L. Alvarez-Gaum\'e and E. Witten, {\em Gravitational anomalies}. { Nucl. Phys}. B, {\bf 234} (1983), 269--330.

\bibitem {A67} M.~F. Atiyah, $K$-theory, Benjamin, New York, 1967.

 \bibitem{AH70} M. Atiyah and F. Hirzebruch, Spin-manifolds and group actions, In, {\it 1970 Essays on Topology and Related Topics (Memoires dedies Georges de Rham)} Springer, New York, 18--28.
 
 \bibitem{AtiyahSchmid} M.~F. Atiyah, W. Schmid, {\em A geometric construction of the discrete series for semisimple
			{L}ie groups,} { Invent. Math.}  {\bf 42} (1977) 1--62.

 \bibitem{BCH94}  P. Baum, A. Connes,  N. Higson, Classifying space for proper actions and K-theory of group C*-algebras, {\it C*-algebras: 1943-1993},   
 { Contemp. Math.}, {\bf 167} (1994) 240--291. 

\bibitem{BL03} L. Borisov and A. Libgober, {\em Elliptic genera of singular varieties}, {Duke Math. J.} {\bf 116} (2003), 319--351.



\bibitem{BT89} R. Bott and C. Taubes,  {\em On the rigidity theorems of Witten}, { J. Amer. Math. Soc.} {\bf 2} (1989), no. 1, 137--186.

\bibitem{CEN} J. Chabert, S. Echterhoff and R. Nest,
{\em The Connes-Kasparov conjecture for almost connected groups and for linear $p$-adic groups.} 
{Publ. Math. Inst. Hautes Etudes Sci.} {\bf 97} (2003), 239--278. 



\bibitem {Ch85} K. Chandrasekharan, {Elliptic Functions}. Springer-Verlag, 1985.

\bibitem {CH} Q. Chen and F. Han, {\em Elliptic Genera, Transgression and Loop Space Chern-Simons Forms}, {Comm. Anal. Geom.}, {\bf 17},  (2009) no.1, 73--106.

\bibitem{De1} A. Dessai, {\em The Witten genus and $S^3$-actions on
manifolds}, {\url{ https://homeweb.unifr.ch/dessaia/pub/papers/MZpreprint6_Witten_S3.pdf}}


\bibitem{De2} A. Dessai, {\em Rigidity for spin$^c$ manifolds}, {Topology}, {\bf 39} (2000), 239--258. 

\bibitem{GMW} H. Guo, V. Mathai and H. Wang, {\em Positive Scalar Curvature and Poincar\'e Duality for Proper Actions}, {J. Noncommut. Geom.}  {\bf 13} (2019) no.4, 1381--1433.


\bibitem{HBJ} F. Hirzebruch, T. Berger and R. Jung, {Manifolds and
Modular Forms.} Aspects of Mathematics, vol. E20, Vieweg, Braunschweig 1992.

\bibitem{HM16} P. Hochs and  {{V. Mathai}}, {\em Spin manifolds and proper group actions},{  Adv. Math.}, {\bf 292} (2016) 1--10. 

\bibitem {Hu} D. Huybrechts, { Complex Geometry, An Introduction}, Springer, 2004. 

\bibitem{Kas83} G. Kasparov, Operator K-theory and its applications: Elliptic operators, group representations, higher signatures, $C^*$-extensions, in: {\it Proc. Internat. Congress of Mathematicians},  Warsaw, {\bf 2} (1983), 987--1000.

\bibitem{Kas88} G. Kasparov, {\em Equivariant KK-theory and the Novikov conjecture}, { Invent. Math.} {\bf 91} (1988), no. 1, 147--201.

\bibitem{Kri90} I.~M. Krichever.
 {\em Generalized elliptic genera and {B}aker-{A}khiezer functions}.
 {Mat. Zametki}, {\bf 47} (1990) no. 2, 34--45, 158.


\bibitem{Lafforgue} V. Lafforgue, {\em K-theorie bivariante pour les algebres de Banach et conjecture de Baum-Connes. }
{ Invent. Math.} {\bf 149} (2002), no. 1, 1--95. 


\bibitem{LafforgueICM} V. Lafforgue, Banach KK-theory and the Baum-Connes conjecture. {\em Proceedings of the International Congress of Mathematicians,} Vol. II (Beijing, 2002), 795--812, Higher Ed. Press, Beijing, 2002.



\bibitem{LS88} P.~S. Landweber and Robert~E. Stong.
{\em Circle actions on {S}pin manifolds and characteristic numbers.}
 {Topology}, {\bf 27}  (1988) no. 2, 145--161.
 
 \bibitem{LM}
 H. Blaine Lawson, Jr. and Marie-Louise Michelsohn. Spin geometry, volume 38
of Princeton Mathematical Series. Princeton University Press, Princeton, NJ,
1989.

\bibitem {Liu95cmp} K. Liu, {\em Modular invariance and characteristic
numbers.} { Commun. Math. Phys}. {\bf 174} (1995), 29--42.


\bibitem{Liu96} K. Liu, {\em On elliptic genera and theta-functions},  { Topology}, {\bf 35}  (1996) no.3, 617--640

\bibitem{Liu95} K. Liu, {\em On modular invariance and rigidity theorems,} {J. Differential Geom.} {\bf 41} (1995) 343--396.

\bibitem{LM1} K. Liu and X. Ma,
{\em  On family rigidity theorems. {I}.}
 {Duke Math. J.}, {\bf 102}  (2000) no.3, 451--474.

\bibitem{LM2} K. Liu and X. Ma,
{\em On family rigidity theorems for {${\rm Spin}^c$} manifolds.}
 In, {Mirror Symmetry, {IV} ({M}ontreal, {QC}, 2000)}, {\bf 33 }, {\em AMS/IP Stud. Adv. Math.}, pages 343--360. Amer. Math. Soc., Providence, RI, 2002.

\bibitem{LMZ1} K. Liu, X. Ma and W. Zhang,
{\em  {${\rm Spin}^c$} manifolds and rigidity theorems in {$K$}-theory.}
 {Asian J. Math.}, {\bf 4} (2000) no. 4, 933--959.
 Loo-Keng Hua: a great mathematician of the twentieth century.

\bibitem{LMZ2} K. Liu, X. Ma and W. Zhang,
{\em Rigidity and vanishing theorems in {$K$}-theory.}
 {Comm. Anal. Geom.}, {\bf 11}  (2003) no. 1, 121--180.

\bibitem{ML} D. McLaughlin, {\em Orientation and string structures on loop space}, {Pacific J. Math.}, {\bf 155}, (1992) no. 1, 143--156.

\bibitem{Phillips} N. Phillips, {\em Equivariant K-theory for proper actions. II. Some cases in which finite-dimensional bundles suffice},  In, Index theory of elliptic operators, foliations, and operator algebras,  {\em Contemp. Math.}, {\bf 70}, (1988) 205--227. 
		


\bibitem{O87} S. Ochanine,  {\em Sur les genres multiplicatifs d\'efinis par des int\'egrales
  elliptiques},  {Topology}, {\bf 26}  (1987) no. 2, 143--151.



\bibitem{T89}  C. Taubes, {\em $S^1$-actions and elliptic genera}, 
{Comm. Math. Phys.}, {\bf 122} (1989), no. 3, 455--526. 

\bibitem{T00} B. Totaro. {\em Chern numbers for singular varieties and elliptic homology}. {Ann. Math.} {\bf 151} (2000), 757-791.

\bibitem{Valette} A. Valette, {\em K-theory for the reduced $C^*$-algebra of a semisimple Lie group with real rank
1 and finite centre}, { Oxford Q. J. Math.} {\bf 35} (1984), 341--359.


\bibitem{W03} C.L. Wang, {\em K-equivalence in birational geometry and characterization of complex elliptic genera}, { J. Alg. Geom.}, {\bf 12} (2003), no. 2, 285--306. 


\bibitem{HW} H. Wang, {\em $L^2$-index formula for proper cocompact group actions}, {J. Noncommut. Geom.} 8 (2014), no. 2, 393--432.

\bibitem{wassermann} A. Wassermann, {\em Une d\'emonstration de la
conjecture de {C}onnes-{K}asparov   pour les groupes de {L}ie
lin\'eaires connexes r\'eductifs}, {C. R. Acad. Sci. Paris S\'er. I},  {\bf 304} (1987), 559--562.

\bibitem{W86} Edward Witten.
 The index of the {D}irac operator in loop space.
 In {\em Elliptic Curves and Modular Forms in Algebraic Topology
  ({P}rinceton, {NJ}, 1986)}, {\bf 1326} of {\em Lecture Notes in Math.},
  pages 161--181. Springer, Berlin, 1988.


\bibitem{W87} E. Witten, {\em Elliptic genera and quantum field theory}, {Comm. Math. Phys.} {\bf 109} (1987), no. 4, 525--536.


\bibitem {Z01} W. Zhang, {Lectures on Chern-Weil Theory and
Witten Deformations.} Nankai Tracts in Mathematics, {\bf 4}, World
Scientific, Singapore, 2001.


\end{thebibliography}
\end{document}